\documentclass[12pt]{amsart}
\usepackage[all]{xypic}
\usepackage{amsmath}
\usepackage{amscd}
\usepackage{amssymb}
\usepackage{amsthm}
\usepackage{enumerate}
\usepackage{amsfonts}
\usepackage{graphicx}
\usepackage[all]{xy}
\usepackage{mathrsfs}
\usepackage[active]{srcltx}
\usepackage{subfigure}
\usepackage{fullpage}

\DeclareFontEncoding{OT2}{}{}

\newcommand{\rar}{\longrightarrow}

\newcommand{\al}{\alpha}
\newcommand{\be}{\beta}
\newcommand{\ga}{\gamma}

\newcommand{\de}{\delta}

\newcommand{\eps}{\epsilon}
\newcommand{\sg}{\sigma}

\newcommand{\bN}{{\mathbb N}}

\newcommand{\bR}{{\mathbb R}}

\newcommand{\cA}{{\mathcal A}}
\newcommand{\cB}{{\mathcal B}}
\newcommand{\cC}{{\mathcal C}}

\newcommand{\cF}{{\mathcal F}}

\newcommand{\cJ}{{\mathcal J}}

\newcommand{\cP}{{\mathcal P}}

\newcommand{\cS}{{\mathcal S}}

\newcommand{\sbr}{\smallbreak}

\def\diam{\mathrm{diam}}

\newcommand{\indlim}[1]{\lim\limits_{\underset{N}{\longrightarrow}}}
\newtheorem{thm}{Theorem}
\newtheorem{cor}[thm]{Corollary}
\newtheorem{lem}[thm]{Lemma}
\newtheorem{conjecture}{Conjecture}

\newtheorem{prop}[thm]{Proposition}
\newtheorem*{t:mainth0}{Theorem}

\theoremstyle{remark}
\newtheorem{rem}{Remark}

\newtheorem{example}{Example}

\newcommand{\bbe}{\begin{equation}}
\newcommand{\ee}{\end{equation}}

\title[From Apollonian packings to
homogeneous sets]{From Apollonian packings to homogeneous sets}

\author{Sergei Merenkov}
\address{Department of Mathematics\\
University of Illinois at Urbana-Champaign\\ 1409 W. Green str.\\ Urbana, IL
61801\\USA\\ 
and Department of Mathematics\\
The City College of New York\\
Convent Ave. at 138th Str.\\
New York, NY 10031\\USA
} \email{merenkov@illinois.edu\\ smerenkov@ccny.cuny.edu}
\thanks{The first author was supported by NSF grant DMS-1001144 and
the second author was supported by NSF grant DMS-0901230.
}

\author{Maria Sabitova}
\address{Department of Mathematics\\
CUNY Queens College\\ 65-30 Kissena Blvd.\\ Flushing, NY
11367\\USA} \email{Maria.Sabitova@qc.cuny.edu}

\date{\today}

\begin{document}

\begin{abstract} 
We extend fundamental results concerning Apollonian packings, which constitute a major object of study in number theory, to certain homogeneous sets that arise naturally in complex dynamics and geometric group theory. 
In particular, we give an analogue of D.~W.~Boyd's theorem (relating the curvature distribution function of an Apollonian packing to its exponent and the Hausdorff dimension of the residual set) for Sierpi\'nski carpets that are Julia sets of hyperbolic rational maps.
%
%



\end{abstract}

\maketitle




\section{Introduction}\label{S:Intro} 
%

One of the most studied objects in number theory which continues to intrigue mathematicians since ancient times is the theory of Apollonian packings. An {\em Apollonian circle packing} can be formed as follows. Consider three mutually exterior-wise tangent circles $C_1$, $C_2$, and $C_3$ in the plane, i.e., each circle touches each of the other two at exactly one point and the open discs enclosed by $C_1,C_2$, and $C_3$ do not intersect pairwise (e.g., see the circles labeled by $18$, $23$, and $27$ in Figure~\ref{pic:Ap} below). A theorem of Apollonius  says that there exist exactly two circles that are tangent to all the three circles $C_1$, $C_2$, and $C_3$. In our example in Figure~\ref{pic:Ap} it would be the big outer circle, denoted by  $C_0$, and the small one, say $C_4$, inscribed inside the curvilinear triangle formed by arcs of $C_1$, $C_2$, and $C_3$. Applying the Apollonius theorem to any three mutually tangent circles among $C_0,C_1,\ldots, C_4$ (for the circle $C_0$ one should take the open disc it bounds as its exterior), one gets new circles inside the disc bounded by $C_0$. Continuing this process indefinitely, one obtains an Apollonian circle packing. 

\sbr
%

\begin{figure}
[htbp] 
\includegraphics[height=40mm]{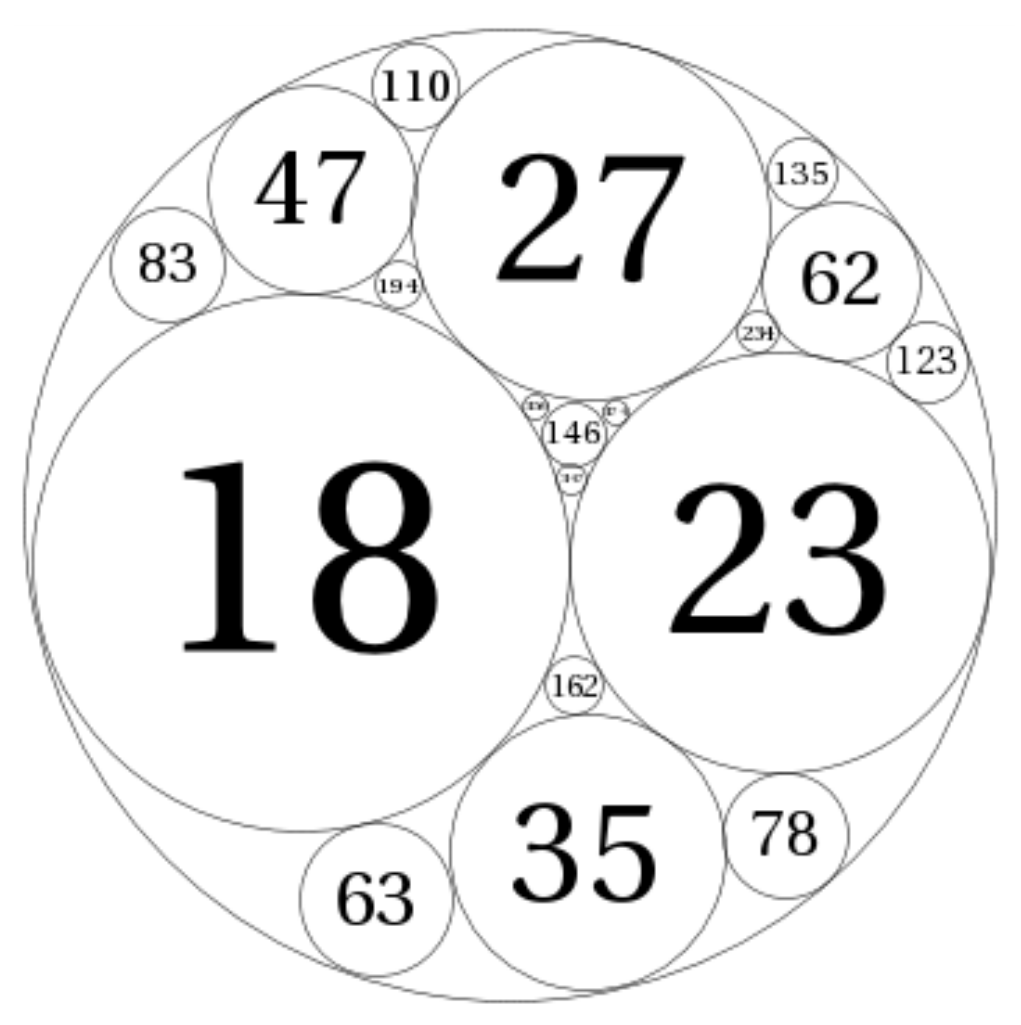}\caption{An Apollonian circle packing with integer curvatures.}
\label{pic:Ap}
\end{figure}

Apollonian circle packings are of interest in number theory in particular due to the observation that 
 if 
four circles $C_0, C_1,C_2,C_3$
 have integer curvatures, then all the circles in the packing also have integer curvatures  
 (in Figure~\ref{pic:Ap} the numbers inside circles are their curvatures and  $C_0$ has curvature $10$). 
There has been extensive research done on number theoretic aspects of Apollonian circle packings as well as on related  objects (see \cite{s2} for a summary of recent advances; also \cite{bf}, \cite{g1}, \cite{g2}, \cite{g3}, \cite{g4},  \cite{k}, 
\cite{o}, \cite{s1}, and others).

\sbr

Consider an Apollonian circle packing $\cP$ whose circles are denoted by $C_0,C_1,\ldots ,$ and for each $k\in\{0,1,\ldots\}$ let $r_k$ denote the radius of $C_k$. With $\cP$ one associates {\em the exponent} $E=E(\cP)$ given by
$$
E=\inf\{t\in\bR\,\,\vert\,\,\sum_kr_k^t<\infty\}.
$$
Note that $E\ne\pm\infty$. Indeed, clearly $t$ is non-negative,  and the series $\sum_kr_k^2$ is convergent since $\pi\sum_kr_k^2$ is bounded by the area enclosed by the outer circle $C_0$. 
One also defines the following {\em curvature distribution function} 
$$
N(x)=\#\{k\,\,\vert\,\, r_k^{-1}\leq x\},\quad x\in\bR.
$$

Let $\cS$ be the complement of the union of all the open discs enclosed by circles $C_k$, $k\in\{1,2,\ldots \}$, in the closed disc enclosed by $C_0$. 
The set $\cS$ is called {\em the  residual set of}\, $\cP$. It is a set of Lebesgue measure zero and is a {\em fractal} in the sense that it behaves similarly on all scales. A natural notion of dimension that can be associated to any metric space, in particular the residual set of an Apollonian circle packing endowed with the restriction of the Euclidean metric, is  the {Hausdorff dimension}. 
It is known that the Hausdorff dimension $\dim_H\cS$ of the residual set $\cS$ of an Apollonian circle packing is approximately $1.305688$ (see,  e.g., \cite{m}, p. 692). The following result is fundamental in the study of Apollonian circle packings and provides the main motivation for the present paper. 
\begin{thm}[D. W. Boyd, \cite{b1}, \cite{b2}]\label{th:2}
If $\cP$ is an Apollonian circle packing, then
$$
\lim_{x\to\infty}\frac{\log N(x)}{\log x}=E=\dim_H \cS,
$$
where $N$ is the curvature distribution function, $E$ is the exponent of $\cP$, and $\dim_H \cS$ is the Hausdorff dimension of the residual set $\cS$ of $\cP$.
\end{thm}

It is natural to ask whether an analogue of Theorem~\ref{th:2} holds for other subsets of the plane that have fractal nature akin to the residual set of an Apollonian circle packing. 
In this paper we propose to investigate this question for a family of sets with controlled geometry on all scales, the so-called homogeneous sets defined in Section~\ref{S:HS}. These are the residual sets of packings by topological (rather than geometric) discs that have roughly rounded shapes, appear at all locations and on all scales, and are relatively separated in the sense that two large elements of the packing cannot be too close to each other. The last condition can be replaced by the assumption that the elements of a packing are uniformly fat. One example of such a homogeneous set is the standard Sierpi\'nski carpet $S_3$ (see Figure \ref{f:ser} below). For a  homogeneous  set we establish asymptotic relationships  between a certain natural analogue of the curvature distribution function and the Minkowski dimensions of the corresponding residual set $\cS$ (Section~\ref{S:DF}). 
It turns out that for a general homogeneous set the exponent of the corresponding packing equals the upper Minkowski dimension of $\cS$ rather than the Hausdorff dimension  of $\cS$. 
As the main application of these results we prove an analogue of Theorem \ref{th:2} for Sierpi\'nski carpets that are the Julia sets of hyperbolic rational maps (Section~\ref{s:jls}).
Finally, in Section~\ref{CR} we state some open problems.

\medskip\noindent 
\textbf{Acknowledgments.} The authors would like to thank Mario Bonk for many useful comments and suggestions.

\section{Homogeneous sets and carpets}\label{S:HS}

Let $\{C_k\}_{k\in\bN\cup\{0\}}$ be a collection of simple closed curves in the plane such that each $C_k,\ k\in\mathbb N$, is enclosed by $C_0$ (i.e., $C_k$ is contained in the closure of the bounded component $D_0$ of the complement of $C_0$)  and suppose that for each pair $j, k\in\mathbb N$ with $j\neq k$, the bounded complementary components $D_j$ and $D_k$ of $C_j$ and $C_k$, respectively, are disjoint. In analogy with Apollonian circle packings, we call the collection of curves $\cP=\{C_k\}_{k\in\bN\cup\{0\}}$, or, interchangeably, the collection of the corresponding topological discs $\{D_k\}_{k\in\bN\cup\{0\}}$, a \emph{packing}. The \emph{residual set} $\cS$ \emph{associated to} such a packing  $\cP$ (or, simply, the {\em residual set of} $\cP$) is the compact subset of the closure $\overline D_0$  obtained by removing from $\overline D_0$ all the domains $D_k,\ k\in\mathbb N$. 
One can similarly define packings and the associated residual sets in the sphere rather than in the plane. In the following it makes no difference whether we consider planar or spherical packings. 
We say that the residual set $\cS$ of a packing $\cP$  is \emph{homogeneous} if it satisfies~properties \eqref{eq:cond1}, \eqref{eq:cond2}, and~\eqref{eq:cond2.5}, or \eqref{eq:cond1}, \eqref{eq:cond2}, and~\eqref{eq:cond3.5} below. 

\sbr

If one hopes to prove an analogue of Theorem~\ref{th:2} for a packing $\cP=\{C_k\}_{k\in\bN\cup\{0\}}$, then it is reasonable to assume  that domains $D_k$'s
corresponding to
 curves $C_k$'s  have roughly rounded shapes. 
More precisely,
\begin{gather}\label{eq:cond1}
\begin{split}
&\text{there exists a constant }\al\geq 1\text{ such that for each }D_k\text{ there exist inscribed }\\ 
&\text{and circumscribed concentric circles of radii }r_k\text{ and }R_k,\text{ respectively, with}\\
\end{split}\\
\frac{R_k}{r_k}\leq\al.\notag
\end{gather}
For example, \eqref{eq:cond1} holds if all $C_k$'s are circles 
or squares.

\sbr

It also seems clear that
 one needs to impose a condition on the residual set $\cS$ associated to $\cP=\{C_k\}_{k\in\bN\cup\{0\}}$ that imitates the fractal nature of the residual set of an Apollonian circle packing.  We consider  the following condition, which is standard in complex dynamics and geometric group theory:
\begin{gather}\label{eq:cond2}
\begin{split}
& \text{there exists a constant }\be\geq 1\text{ such that for any }p\in \cS\text{ and }r,\ 
0<r\leq\diam\,\cS,\\
&\text{there exists a curve } C_k\text{  satisfying }
C_k\cap B(p,r)\ne \emptyset\text{ and }\\
\end{split}\\
\frac{1}{\be}r\leq \diam \, C_k\leq \be r. \notag
\end{gather}
Here and in what follows $B(p,r)$  denotes the open disc of radius $r$ centered at $p$, and $\diam\, X$ stands for the diameter of a metric space $X$.  Geometrically, property \eqref{eq:cond2} means that curves $C_k$'s appear at all locations and on all scales. 
\begin{rem}
Note that properties~\eqref{eq:cond1} and~\eqref{eq:cond2} readily imply that the Lebesgue measure of $\cS$ is zero.  Indeed, otherwise $\cS$ would contain a Lebesgue density point, which clearly contradicts to~\eqref{eq:cond1} and~\eqref{eq:cond2}. 
\end{rem}

Further, we require the following condition of relative separation,  also standard in complex dynamics and geometric group theory: 
\begin{gather}\label{eq:cond2.5}
\begin{split}
&\text{there exists a constant }\delta>0\text{ such that for any }j\neq k\text{ we have }\\
\end{split}\\
\Delta(C_j, C_k):=\frac{{\rm dist}(C_j,C_k)}{\min\{\diam\, C_j,\diam\, C_k\}}\ge\delta. \notag
\end{gather}
Here ${\rm dist}(C_j, C_k)$ stands for the distance between $C_j$ and $C_k$. One can visualize property \eqref{eq:cond2.5} as forbidding
two large curves $C_j$ and $C_k$ to be quantitatively too close to each other.

\sbr

\begin{rem}
While extending analysis from Apollonian circle packings to more general packings, one would like to recoup an important property of rigidity. Conditions \eqref{eq:cond1}, \eqref{eq:cond2}, and \eqref{eq:cond2.5} can be considered as a way of doing so.

\sbr

%
Clearly, Apollonian circle packings satisfy \eqref{eq:cond1}. However, they do not satisfy either \eqref{eq:cond2} or~\eqref{eq:cond2.5}. 
Indeed,  the failure of~\eqref{eq:cond2.5} is immediate because in Apollonian circle packings the circles are allowed to touch. We now show that~\eqref{eq:cond2} fails. Let $C_1$ and $C_2$ be two tangent circles in an Apollonian circle packing having disjoint interiors and let $p$ denote the common point of $C_1$ and $C_2$.
Assume $\epsilon_k$ is a sequence of positive numbers with $\epsilon_k\to 0$ as $k\to\infty$. For each $k$ denote by $r_k$ the radius of a circle from the packing other than $C_1$ or $C_2$ that intersects $B(p,\epsilon_k)$ and has the largest radius. It can be checked that the sequence $r_k$ goes to zero faster than $\epsilon_k$ when $k\to\infty$ (in fact, $r_k$ is majorized by $\epsilon_k^2$). Thus homogeneous sets do not generalize Apollonian circle packings but rather complement them.
\end{rem}

As pointed out in the previous paragraph,  condition~\eqref{eq:cond2.5}  does not allow for two elements of a packing to touch. It turns out that in the main results of the paper this condition can be replaced by the following co-fatness condition for $\cS$ (see~\cite{S95}):
\begin{gather}\label{eq:cond3.5}
\begin{split}
&\text{there exists a constant }\tau>0\text{ such that for every }k\text{ and each disc }B(p,r)\\
&\text{that is centered on }D_k\text{ and does not contain }D_k,\text{ we have }\\
\end{split}\\
{\rm area}(D_k\cap B(p,r))\ge\tau r^2.\notag
\end{gather}

\begin{rem}
It is easy to check that neither of \eqref{eq:cond2.5} or \eqref{eq:cond3.5} implies the other. Condition~\eqref{eq:cond3.5} is often easier to check, e.g, as \eqref{eq:cond1} it holds in the case when all $C_k$'s are circles or squares. Moreover, the Sierpi\'nski gasket (see Figure~\ref{f:gask}) satisfies \eqref{eq:cond3.5} but not \eqref{eq:cond2.5}. It is not hard to see that the Sierpi\'nski gasket also satisfies~\eqref{eq:cond1} and~\eqref{eq:cond2}, and so it is homogeneous.
\begin{figure}
[htbp]
\includegraphics[height=35mm]{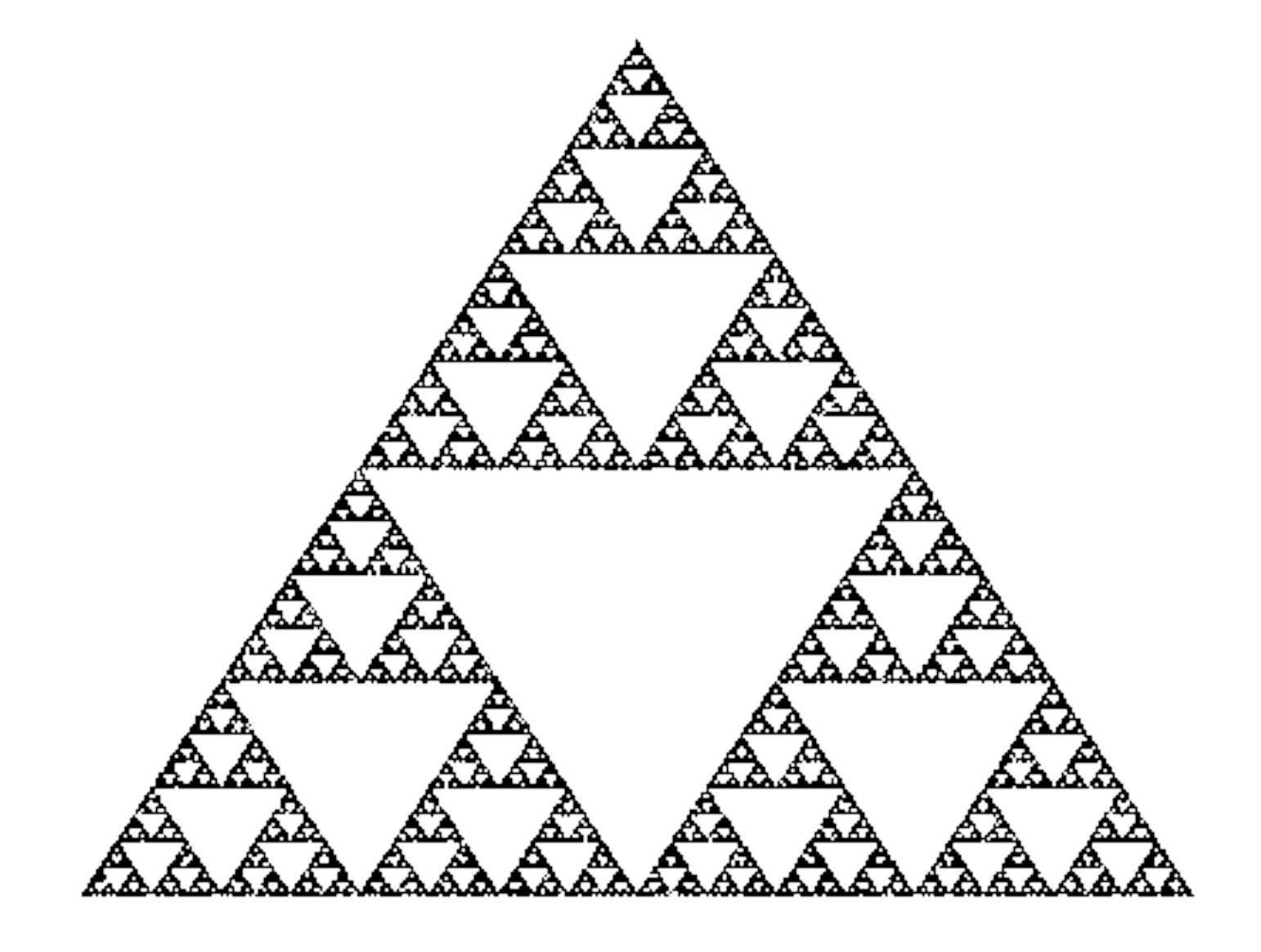}\caption{The Sierpi\'nski gasket.}\label{f:gask}
\end{figure}

\end{rem}

\begin{example}\label{ex:st}
Standard Sierpi\'nski carpets are homogeneous. Recall that for an odd integer $p>0$ the {standard Sierpi\'nski} $p$-{carpet} $S_p$ (defined up to translations and scalings of the plane) is obtained as follows.
As the first step of the construction we subdivide  
a closed square in the plane
into $p^2$ equal subsquares in the obvious way
and then remove the interior of the middle square (the middle square is well defined because $p$ is odd). In the second step we perform the same operations (subdivide into $p^2$ equal subsquares and remove the interior of the middle square) on the $p^2-1$ subsquares remaining after the first step. If the process is continued indefinitely, then what is 
left is called the {\em standard Sierpi\'nski $p$-carpet} (see Figure \ref{f:ser} below for $p=3$). A proof that $S_3$ and $S_5$ are homogeneous is contained in Example~\ref{eq:weird} below. For a general $p$ the arguments are similar.  
\begin{figure}
[htbp]
\includegraphics[height=35mm]{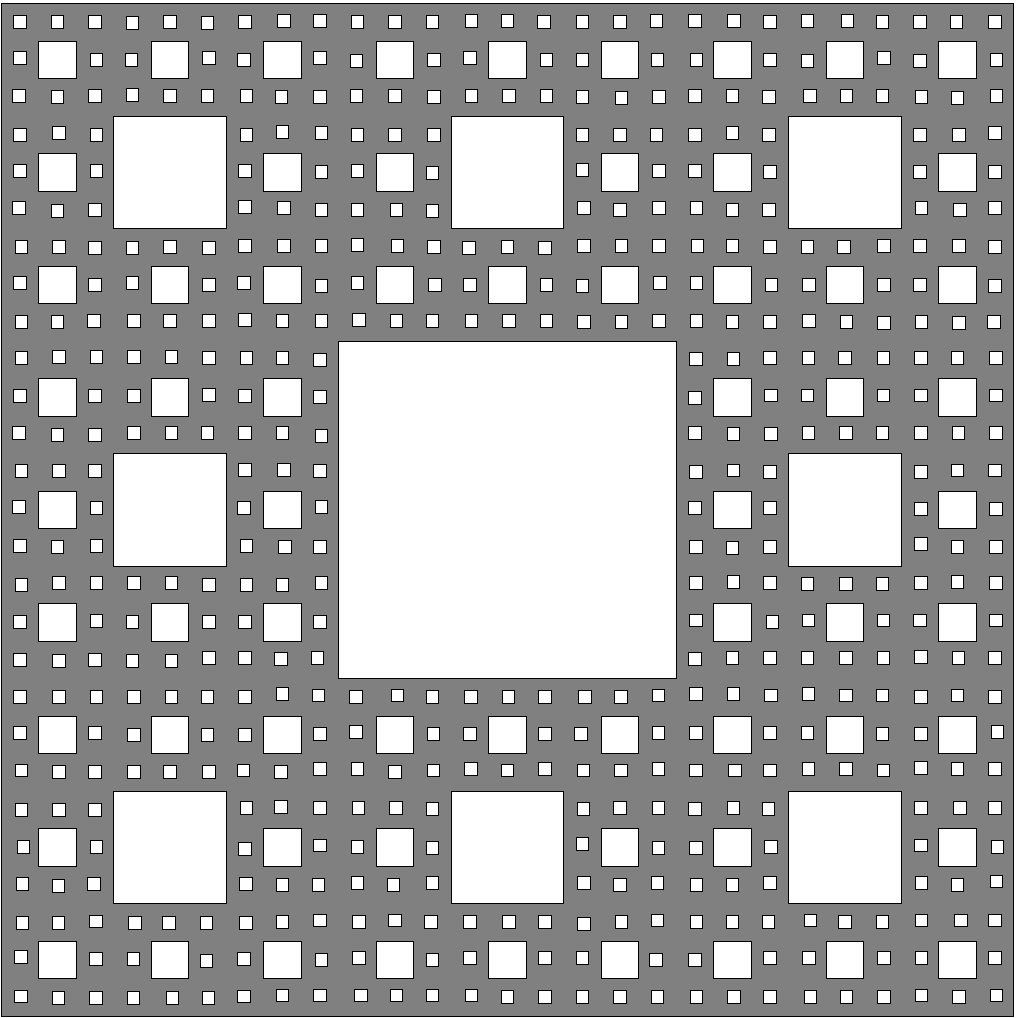}\caption{The standard Sierpi\'nski 3-carpet $S_3$.}\label{f:ser}
\end{figure}
\end{example}

By definition,
a \emph{Sierpi\'nski carpet}, or \emph{carpet} for short,  is any topological space that is homeomorphic to $S_3$. A \emph{peripheral circle} of a carpet $S$ is any simple closed curve in $S$ that corresponds under a homeomorphism to the boundary of one of the removed squares in the construction of $S_3$.  Whyburn's characterization~\cite{W58} states that a metrizable topological space is a carpet if and only if it is a planar continuum of topological dimension 1 that is locally connected and has no local cut points. Here a \emph{local cut point} is a point whose removal separates  the space locally. This gives a way to produce a large supply of carpets as  follows. Suppose that  $\overline D_0$ is a closed topological disc in the plane or in the sphere and $\{D_k\}_{k\in\mathbb N}$ are pairwise disjoint open topological discs contained in $D_0$ that are  bounded by simple closed curves.  
Then the space $\cS=\overline D_0\setminus\cup_{k\in\mathbb N} D_k$ is a carpet if and only if $\cS$ has no interior, $\diam\, D_k\to0$ as $k\to\infty$, and for each pair $j, k\in\mathbb N\cup\{0\},\ j\neq k$, the boundaries $\partial D_j$ and $\partial D_k$ are disjoint. In this way, any carpet is represented naturally as the residual set associated to a packing $\cP=\{\partial D_k\}_{k\in\mathbb N\cup\{0\}}$, where the boundaries of $D_k$'s are peripheral circles.
As a partial converse, any homogeneous residual set $\cS$ satisfying~\eqref{eq:cond2.5} is a carpet. Indeed, \eqref{eq:cond1} implies $\diam\, D_k\to0$ as $k\to\infty$, \eqref{eq:cond2} gives that $\cS$ has no interior,  and~\eqref{eq:cond2.5} implies that $\partial D_j$ and $\partial D_k$ are disjoint for all $j\neq k$.

\begin{example}\label{ex:jul}
If the Julia set $\cJ(f)$ of a hyperbolic rational map $f$ is a carpet, then it is homogeneous (see the proof of  Theorem~\ref{th:Jul} below; definitions of a Julia set and hyperbolicity can be found in
Section~\ref{s:jls} below). An example of such a map is $f(z)=z^2-1/(16 z^2)$ (see Figure~\ref{f:jul} for its Julia set). 


\begin{figure}
[htbp]
\includegraphics[height=35mm]{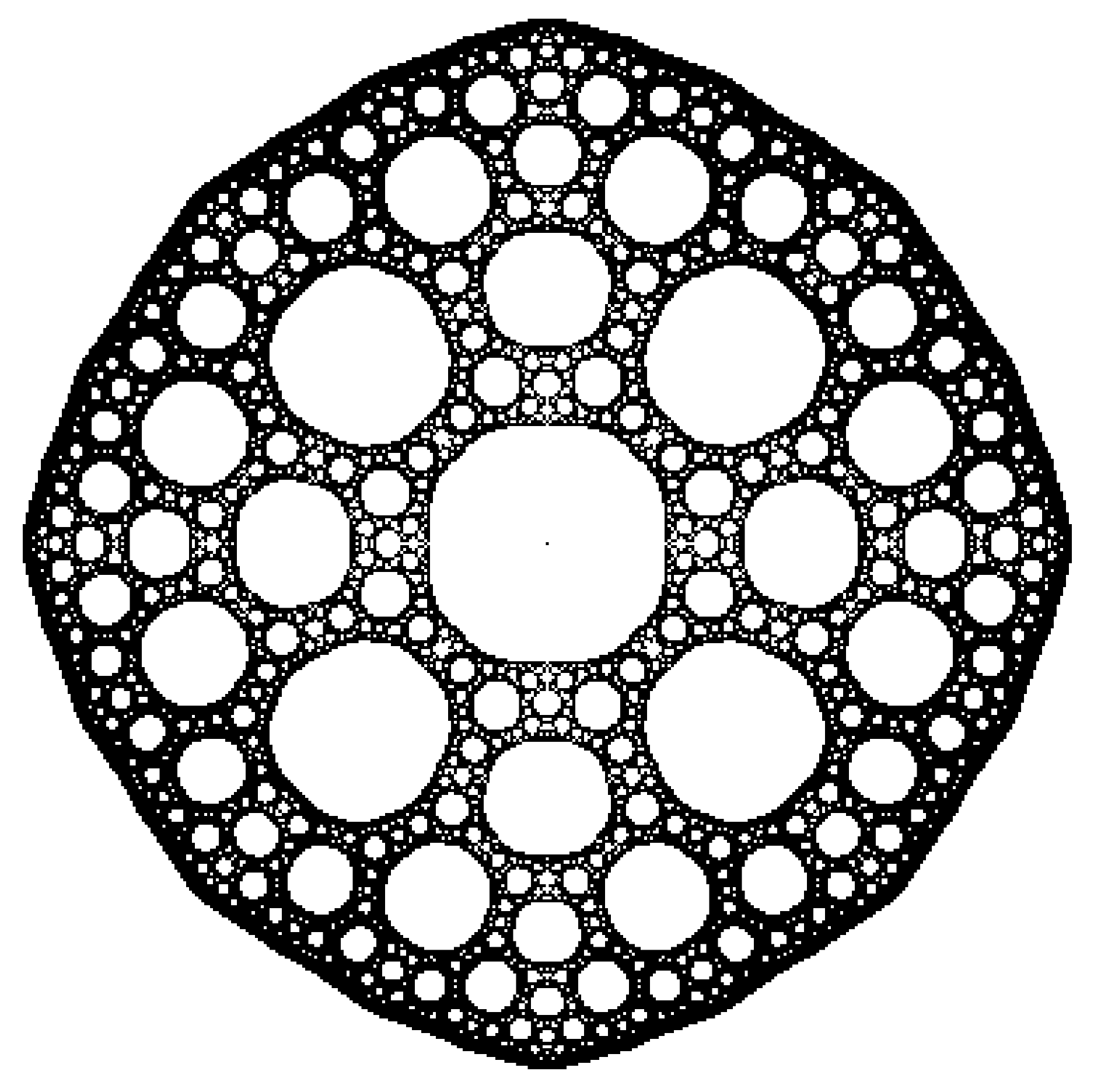}\caption{The Julia set of $f(z)=z^2-\frac1{16z^2}$.}\label{f:jul}
\end{figure}

It is known that if the limit set $\mathcal L(G)$ of a convex-cocompact Kleinian group $G$ is a carpet, then it is homogeneous. (This can be extracted from~\cite{KK00}.) In particular, if $G$ is the fundamental group of a compact hyperbolic 3-manifold with  non-empty totally geodesic boundary, then its limit set $\mathcal L(G)$ is a homogeneous carpet. We do not prove these statements as we do not need them. For a comprehensive  treatment of asymptotics of the curvature distribution function for    circle packings invariant under non-elementary Kleinian groups see~\cite{o}.
\end{example}
Assume that $\cP=\{C_k\}_{k\in\bN\cup\{0\}}$ is a packing whose associated residual set $\cS$ satisfies~\eqref{eq:cond1}. Similarly to the case of Apollonian circle packings we define the \emph{exponent} $E=E(\cP)$ of $\cP$ as
$$
E=\inf\{t\in\bR\,\,\vert\,\,\sum_{k}(\diam\,C_k)^t<\infty\}.
$$
Let $a_k=\diam\,C_k$ and for a natural $m$ denote $M(t,m)=\sum_{k=1}^ma_k^t$, $M(t)=\sum_{k}a_k^t$. As in the case of Apollonian circle packings we have $M(2)<\infty$. Indeed, using \eqref{eq:cond1} we have
$$
\pi\sum_ka_k^2\leq 4\pi\sum_kR_k^2\leq 4\al^2\sum_k\pi r_k^2\leq 4\al^2\text{Area}(D_0),
$$
since the circles inscribed in $C_k$'s are pairwise disjoint.
Also, without loss of generality we can assume that all $a_k\le1$, and thus for $t\geq t'\ge0$ we have $M(t,m)\leq M(t',m)$ for any $m$. Finally, $M(0)=\infty$. These imply that $E<\infty$ and 
$$
E=\inf\{t\in\bR\,\,\vert\,\,\sum_{k}(\diam\,C_k)^t<\infty\}=
\sup\{t\in\bR\,\,\vert\,\,\sum_{k}(\diam\,C_k)^t=\infty\}.
$$

The \emph{curvature distribution function} $N(x)$ of $\cP$ is defined by
$$
N(x)=\#\{k\,\,\vert\,\, (\diam\, C_k)^{-1}\leq x\},\quad x\in\bR.
$$
It turns out that for a general homogeneous residual set $\cS$ associated to a packing $\cP$ the limit of ${\log N(x)}/{\log x}$ does not exist when $x$ tends to infinity, i.e.,
\bbe
\label{eq:oops}
\limsup_{x\to\infty}\frac{\log N(x)}{\log x}\ne\liminf_{x\to\infty}\frac{\log N(x)}{\log x}.
\ee
Also,  the Hausdorff dimension of $\cS$ is not necessarily  equal to $E(\cP)$ (see Example~\ref{eq:weird} below). This contrasts Theorem \ref{th:2}.
\begin{example}\label{eq:weird}
Using the idea of Sierpi\'nski carpets we now construct the following packing for which \eqref{eq:oops} holds. To a closed square in the plane we apply the first $n_1$ steps of the construction of the $3$-carpet. Then, to each of the remaining subsquares we apply the first $n_2$ steps of the construction of the $5$-carpet. After that, to each of the remaining subsquares we apply the first $n_3$ steps of the construction of the $3$-carpet again. This way, alternating, we continue indefinitely. For each sequence $\sg=\{n_1,n_2,n_3,\ldots\}$ we therefore have  a packing $\cP=\cP(\sg)$ and the associated residual set $\cS=\cS(\sg)$ depending on $\sg$. 
In this construction we allow some $n_k$ to be infinite, in which case we assume that the sequence $\sigma$ is finite and its last element is infinity. We also allow the first $n_1$ steps to be the steps of the construction of the 5-carpet. 
Whyburn's characterization implies that $\cS$ is a carpet. 

\sbr

We now check that for any $\sg$ the set $\cS=\cS(\sg)$ is homogeneous, namely it satisfies conditions~\eqref{eq:cond1}, \eqref{eq:cond2}, and~\eqref{eq:cond2.5}. In fact, it also clearly satisfies~\eqref{eq:cond3.5}. Indeed, \eqref{eq:cond1} is trivially satisfied, since each bounded complementary component of $\cS$ is a square. To check~\eqref{eq:cond2} we fix arbitrary $p\in\cS$ and $0<r\le\diam\, \cS$. After performing a number of steps of dividing and removing subsquares as explained in the previous paragraph we will denote by $Q$ one of the biggest subsquares that turn out to be inside $B(p,r)$. Assume also that $Q$ is not a complementary component, i.e., $Q$ is not a ``hole".
More precisely, suppose that $n=n_1n_2\dots n_{k-1}n_k'$, where $1\le n_k'\le n_k$, is the smallest number such that one of the remaining subsquares $Q$ of side length $s=3^{-n_1}5^{-n_2}\dots 3^{-n_k'}$ or $3^{-n_1}5^{-n_2}\dots 5^{-n_k'}$, depending on whether $k$ is odd or even, 
is contained in $B(p,r)$. From the minimality of $n$ it is immediate that $r/\beta'\le s\le \beta'r$, where $\beta'\ge1$ is an absolute constant. Now we choose $C$ to be the boundary of the middle square in the subdivision of $Q$. Then 
$$
\diam\, C=\frac{\sqrt{2}s}{3}\quad\text{ or }\quad \frac{\sqrt{2}s}{5},
$$
and~\eqref{eq:cond2} follows with $\beta=5\beta'/\sqrt{2}$. Finally, to verify~\eqref{eq:cond2.5} let $C_j$ and $C_k$ be the boundaries of two distinct complementary squares in the construction of $\cS$. We may assume that $\diam\, C_j\le\diam\, C_k$. 
Let $Q$ be the subsquare in the construction of $\cS$ so that $C_j$ is the boundary of the middle square in the subdivision of $Q$. Because $\diam\, C_j\le\diam\, C_k$, the curve $C_k$ does not intersect the interior of $Q$ and hence 
$$
\Delta(C_j,C_k)=\frac{{\rm dist}(C_j,C_k)}{\diam\, C_j}\ge\frac{{\rm dist}(C_k,\partial Q)}{\diam\, C_j} \ge\frac{1}{\sqrt{2}},
$$
i.e., \eqref{eq:cond2.5} follows with $\delta=1/\sqrt{2}$. So,  any $\cS$ in the above construction satisfies~\eqref{eq:cond1}, \eqref{eq:cond2}, \eqref{eq:cond2.5}, and~\eqref{eq:cond3.5} and, in fact, the constants $\alpha, \beta, \delta, \tau$ do not depend on  sequence $\sigma$. 

\sbr

If  $\cS$ is the  3-carpet, then one can easily check that
$$\lim_{x\to\infty}\frac{\log N(x)}{\log x}=\frac{\log8}{\log3}.
$$ 
Likewise, if $\cS$ is the 5-carpet, then 
$$
\lim_{x\to\infty}\frac{\log N(x)}{\log x}=\frac{\log24}{\log5}.
$$
It is an elementary exercise to show that there exists a sequence $\sg$ converging to $\infty$ fast enough so that 
$$
\limsup_{x\to\infty}\frac{\log N(x)}{\log x}=\frac{\log 24}{\log 5}\quad\text{ and }\quad
\liminf_{x\to\infty}\frac{\log N(x)}{\log x}=\frac{\log 8}{\log 3}<\frac{\log 24}{\log 5}.
$$
Indeed, given a sequence $\{\eps_1,\eps_2,\ldots \}$ of positive numbers converging to zero it is possible to find a large enough $n_1$ so that for any $n'_2,n'_3,\ldots$ the curvature distribution function $N(x)$ corresponding to $\cS(\sg)$, $\sg=\{n_1,n'_2,n'_3,\ldots \}$, satisfies $\log N(x_1)/\log x_1<\log8/\log3+\epsilon_1$ for some large $x_1$. 
Similarly, it is possible to find a large enough $n_2$ so that for any $n'_3,n'_4,\ldots$ the curvature distribution function $N(x)$ corresponding to $\cS(\sg)$, $\sg=\{n_1,n_2,n'_3,n'_4,\ldots \}$, satisfies
$\log N(x_2)/\log x_2>\log24/\log5-\epsilon_2$ 
for some $x_2>x_1$, and so on. 
We leave the (easy) details to the reader. 

Also,  the Hausdorff dimension is always at most the lower Minkowski dimension (see the definition of the latter below). Thus, according to Proposition~\ref{p:dual} below 
for such a sequence $\sigma$ and the corresponding set $\cS$ one has 
$$\dim_H\cS\le\liminf_{x\to\infty}\frac{\log N(x)}{\log x}=\frac{\log8}{\log 3},$$ where $\dim_H\cS$ denotes the Hausdorff dimension of $\cS$. Theorem~\ref{th:111} below in addition shows that $$E(\cP)=\limsup_{x\to\infty}\frac{\log N(x)}{\log x}=\frac{\log24}{\log5}.$$
\end{example}

\section{Curvature distribution function and  Minkowski dimensions}\label{S:DF}

As we saw in Example \ref{eq:weird} above,
for a general homogeneous residual  set $\cS$, i.e., the residual set $\cS$ of a packing $\cP$ satisfying \eqref{eq:cond1}, \eqref{eq:cond2}, and~\eqref{eq:cond2.5}, or \eqref{eq:cond1}, \eqref{eq:cond2}, and~\eqref{eq:cond3.5}, $E$ is not equal to the Hausdorff dimension $\dim_H\cS$ of $\cS$ in Theorem~\ref{th:2}. 
We show that the right analogue is the  upper Minkowski (or the upper box counting) dimension of $\cS$, which is another notion of dimension applied to fractals. Suppose $n(\eps)$ denotes the maximal number of disjoint open discs of radii $\eps>0$ centered on $\cS$.
Then the \emph{Minkowski dimension} $\dim_{box}\cS$ of $\cS$ is defined as
$$
\dim_{box}\cS=\lim_{\eps\to 0}\frac{\log n(\eps)}{\log(1/\eps)},
$$
provided the limit exists. In general,
this limit does not exist  and in that case one defines the \emph{upper  Minkowski dimension}  $\dim_{ubox}\cS$ (respectively, lower Minkowski dimension $\dim_{lbox}\cS$) as the corresponding upper limit (respectively, lower limit). 
Note that there are other equivalent definitions of Minkowski dimensions, e.g., where $n(\epsilon)$ is the minimal number of open discs of radii $\epsilon$ required to cover the set $\cS$, or $n(\epsilon)$ is the minimal number of open discs of radii $\epsilon$ centered on $\cS$ required to cover $\cS$, etc.  

\sbr

The following proposition is a certain duality result between a  packing and the associated residual set
in the case when the latter is homogeneous.
\begin{prop}\label{p:dual}
If the residual set 
$\cS$ of a packing $\cP$ is homogeneous, then
$$
\begin{aligned}\label{eq:dual}
\limsup_{x\to\infty}\frac{\log N(x)}{\log x}&=\dim_{ubox}\cS,\\
\liminf_{x\to\infty}\frac{\log N(x)}{\log x}&=\dim_{lbox}\cS,
\end{aligned}
$$
where $N(x)$ is the curvature distribution function of $\cP$.
\end{prop}
We will need the following lemma.
\begin{lem}\label{l:bnds}
Assume that $\cS$ is the residual set associated to a packing $\cP=\{C_k\}_{k\in\bN\cup\{0\}}$ satisfying~\eqref{eq:cond1} and~\eqref{eq:cond2.5} {\rm (}or~\eqref{eq:cond1} and~\eqref{eq:cond3.5}{\rm)}. Given any $\beta>0$
there exist constants $\ga_1=\ga_1(\be)\ge1$ that depends only on  $\beta$  and $\ga_2=\ga_2(\al,\be,\de)\ge1$ {\rm(}$\ga_2=\ga_2(\al,\be,\tau)\ge1${\rm)} that depends only on $\alpha$ in \eqref{eq:cond1}, $\beta$, and $\delta$ in ~\eqref{eq:cond2.5} {\rm(}that depends only on $\alpha$ in \eqref{eq:cond1}, $\beta$, and $\tau$ in ~\eqref{eq:cond3.5}{\rm)} such that 
for any collection $\mathcal C$ of disjoint open discs of radii $r$ centered on $\cS$ we have the following.
\begin{gather}\label{eq:cond3}
\begin{split}
\text{There are at most }\ga_1 \text{ discs in }\mathcal C \text{ that intersect any given } C_k\text{ with }\\
\end{split}\\
\diam \,C_k\leq \be r \notag,
\end{gather}
and
\begin{gather}\label{eq:cond4}
\begin{split}
\text{there are at most }\ga_2 \text{ curves } 
C_k \text{ intersecting any given disc in } 2\mathcal C\text{ and satisfying }\\
\end{split}\\
\frac{1}{\be}r\leq \diam \,C_k, \notag
\end{gather}
where $2\mathcal C$ denotes the collection of all open discs with the same centers as those in $\mathcal C$, but whose radii are $2r$.
\end{lem}
\begin{proof}
By rescaling the Euclidean metric by $1/r$, we conclude that to find the bound $\gamma_1$ in~\eqref{eq:cond3} is the same as to find an upper bound on the number of disjoint open discs in the plane of radii 1 that intersect a given set $S$ of diameter at most $\beta$. 
If $B$ is a disc centered at a point of $S$ and whose radius is $\beta+2$, then $B$ contains $S$ as well as any disc of radius 1 that intersects $S$. 
Therefore, if $n$ is the number of disjoint open discs of radii 1 that intersect $S$, then by comparing areas we obtain  $n\le(\beta+2)^2$.

\sbr

To prove~\eqref{eq:cond4} we first observe that~\eqref{eq:cond2.5} or~\eqref{eq:cond3.5} imply the existence of constants $\beta'>0$ and $\nu\in\mathbb N$ such that the number of
$C_k$'s that intersect a given $B(p,2r)\in 2\cC$ and $\diam\, C_k>\beta' r$ is at most $\nu$. 
Indeed, if~\eqref{eq:cond2.5} holds, let $\beta'=4/\delta$ and suppose that there are two distinct curves $C_j, C_k$ that intersect $B(p, 2r)$ and  $\min\{\diam\, C_j, \diam\, C_k\}>\beta'r$. Then ${\rm dist}(C_j, C_k)\le4r$ and hence
$$
\Delta(C_j, C_k)<\frac{4r}{\beta'r}=\delta,
$$
which contradicts~\eqref{eq:cond2.5}, and therefore $\nu=1$. Note that in this case $\beta'$ depends only on $\delta$.

Now assume that~\eqref{eq:cond3.5} holds and let $\beta'=7$. For each $k$ such that $C_k$  intersects $B(p,2r)$ and $\diam\, C_k>\beta' r$,  the corresponding $D_k$ is not contained in $B(p,3r)$. Let $B_k=B(q,r)$ be a disc centered at $q\in D_k\cap B(p,2r)$. Then $B_k$ is contained in $B(p,3r)$, and hence does not contain $D_k$.  Therefore by~\eqref{eq:cond3.5} we have
$$
{\rm area}(D_k\cap B_k)\ge\tau r^2.
$$
Since any two distinct $D_j$ and $D_k$ are disjoint and all $B_k$'s are contained in $B(p,3r)$, for the number $\nu$ of ``large" $C_k$'s as above we have
$$
\nu\tau r^2\le\sum_k{\rm area}(D_k\cap B_k)\le{\rm area}(B(p,3r))=9\pi r^2.
$$
This gives $\nu\le9\pi/\tau$, a bound that  depends only on $\tau$.

Thus, to prove the existence of $\gamma_2$ it is enough to find a bound on  the number of $C_k$'s that intersect a given disc in $2\cC$ and satisfy
\bbe\label{eq:vspom}
\frac1{\beta}r\le\diam\, C_k\le\beta' r.
\ee 
Recall that by~\eqref{eq:cond1} there exists $\alpha\ge 1$ such that for any $C_k$ there exist inscribed and circumscribed concentric circles of radii $r_k$ and $R_k$, respectively, with $R_k/r_k\le\alpha$. Since the interiors $D_k$'s of distinct $C_k$'s are disjoint, the corresponding inscribed discs are disjoint. Moreover, 
for $C_k$ satisfying \eqref{eq:vspom} we have the following inequalities:
$$
\frac1{\beta} r\le \diam\, C_k\le 2R_k\le 2\alpha r_k\le\alpha\, \diam\, C_k\le\alpha\beta'r.
$$
The number of curves $C_k$ that intersect a given disc $B(p,2r)$ in $2\cC$ is bounded above by the number of corresponding circumscribed discs that intersect the same disc $B(p,2r)$. 
Therefore, to find a bound $\ga_2$ in~\eqref{eq:cond4} we need to find an upper bound on the number of disjoint open discs $B(p_k,r_k)$ of radii $r_k$  at least $r/(2\alpha\beta)$ and at most $\beta'r/2$ such that $B(p_k, \alpha r_k)$ intersects $B(p,2r)$ (recall that $R_k\le\alpha r_k$). Rescaling the Euclidean metric by $1/r$ as above, it is equivalent to finding an upper bound  on the number $n$ of disjoint open discs $B(p_k, r_k)$ of radii $r_k$ with $1/(2\alpha\beta)\le r_k\le \beta'/2$ such that $B(p_k,\alpha r_k)$ intersects $B(p,2)$. Similar to the proof of~\eqref{eq:cond3}, each disc $B(p_k,\alpha r_k)$  intersecting $B(p,2)$ is contained in $B(p, 2+\alpha\beta')$. Since $B(p_k, r_k)$'s are disjoint, we obtain  
$$
\sum_{i=1}^n r_k^2\le(2+\alpha\beta')^2.
$$ 
On the other hand,
$$
\sum_{i=1}^n r_k^2\ge\frac{n}{(2\alpha\beta)^2},
$$
and thus
$
n\le(2\alpha\beta(2+\alpha\beta'))^2,
$
which completes the proof of~\eqref{eq:cond4}.
\end{proof}

\begin{proof}[Proof of Proposition~$\ref{p:dual}$]
Let $\epsilon>0$ be arbitrary and let $\mathcal C$ be a maximal collection of disjoint open discs of radii $\epsilon$ centered on $\cS$. 

\sbr

We first establish an upper bound for $n(\eps)$ in terms of the function $N(x)$. For that one can define a function $f$ from $\cC$ to the set $\cA$ of all $C_k$'s with $\epsilon/\beta\le\diam\, C_k\le\beta\epsilon$ by assigning to each disc in $\cC$ a $C_k\in\cA$ intersecting the disc. Note that by \eqref{eq:cond2} every disc in $\mathcal C$ intersects at least one $C_k$ from $\cA$, so that $f$ is defined on the whole set $\cC$. Also, 
given any $C_k$ with $\diam\, C_k\le\beta\epsilon$, by~\eqref{eq:cond3}, there are at most $\gamma_1$ discs in $\mathcal C$ that intersect $C_k$.  
Recall that $n(\eps)$ is the number of elements in $\cC$, and the number of elements in $\cA$ is at most $N(\be/\eps)$. Hence,
comparing the sizes of $\cC$ and $\cA$ via $f:\cC\rar\cA$, we immediately obtain
$$
n(\epsilon)\le\gamma_1 N(\beta/\epsilon).
$$
Therefore,
$$
\dim_{ubox}\cS=\limsup_{\eps\to 0}\frac{\log n(\eps)}{\log(1/\eps)}\le\limsup_{\eps\to 0}\frac{\log(\ga_1 N(\beta/\epsilon))}{\log(1/\eps)}=
\limsup_{x\to\infty}\frac{\log N(x)}{\log x},
$$
and similarly,
$$
\dim_{lbox}\cS\le\liminf_{x\to\infty}\frac{\log N(x)}{\log x}.
$$

We now use an argument similar to the one in the preceding paragraph to obtain an upper bound for $N(x)$ in terms of $n(\eps)$. Namely, we define a function $g$ from the set $\cB$ of all $C_k$'s satisfying $\epsilon/\beta\le \diam\, C_k$ to $2\cC$ by assigning to each $C_k\in\cB$ a disc in $2\cC$ intersecting $C_k$. By the maximality of $\mathcal C$, the collection $2\mathcal C$  covers $\cS$ and hence $g$ is defined on the whole set $\cB$.
 By~\eqref{eq:cond4}, given a disc $B(p,2\eps)\in 2\mathcal C$ there are at most $\ga_2$ curves $C_k$ from $\cB$ intersecting $B(p,2\eps)$. Recall that $N(\beta/\epsilon)$ is the number of elements in $\cB$ and $n(\eps)$ is the number of elements in $\cC$, which is the same as the number of elements in $2\cC$. Hence,
comparing the sizes of $\cB$ and $2\cC$ via $g:\cB\rar2\cC$, we obtain
$$
N(\beta/\epsilon)\le\ga_2 n(\epsilon).
$$
Therefore,
$$
\dim_{ubox}\cS=\limsup_{\eps\to 0}\frac{\log n(\eps)}{\log(1/\eps)}\ge\limsup_{\eps\to 0}\frac{\log(N(\beta/\epsilon)/\ga_2)}{\log(1/\eps)}=
\limsup_{x\to\infty}\frac{\log N(x)}{\log x},
$$
and likewise
$$
\dim_{lbox}\cS\ge\liminf_{x\to\infty}\frac{\log N(x)}{\log x}.
$$
The desired equalities are thus proved.
\end{proof}

We finish this section with the following general theorem whose proof follows the lines of~\cite[p.~250]{b2} or~\cite[p.~126, Theorem~3]{w}.
\begin{thm}\label{th:111}
Let $\cP=\{C_k\}_{k\in\mathbb N\cup\{0\}}$ be a packing such that the associated residual set $\cS$ satisfies~\eqref{eq:cond1}. Then
$$
\begin{aligned}\label{eq:11}
\limsup_{x\to\infty}\frac{\log N(x)}{\log x}=E,
\end{aligned}
$$
where $E$ is the exponent of $\cP$.
\end{thm}
\begin{proof}
Assume that $C_k$'s are numbered in such a way that the sequence $\{\diam\, C_k\}$ is monotone decreasing.  The limit of this sequence must be 0 because of~\eqref{eq:cond1}. Denote $r_k=\diam\, C_k$ and
$$
M=\limsup_{x\to\infty}\frac{\log N(x)}{\log x}.
$$
If $t>E$, then $\sum_kr_k^t$ is convergent and for any $x\in\bR$ we have
$$
x^{-t}N(x)\le\sum_k{r_k^t}<\infty.
$$
Taking $\log$ and the upper limit when $x\to\infty$, and then letting $t\to E_+$, we get $M\leq E$. Assume for the sake of contradiction that $M<E$ and let $t$ satisfy $M<t<E$. Note that there exists $K\in\bR$ such that for any $x\geq K$ we have
$$
t\geq \frac{\log N(x)}{\log x}.
$$
Also, there exists $L\in\bN$ such that for any $k\geq L$ and $x$ with $r_{k}^{-1}\leq x<r_{k+1} ^{-1}$ we have $x\geq K$. Then $N(x)=k$, hence
$$
t\geq \frac{\log N(x)}{\log x}>\frac{\log k}{\log r_{k+1}^{-1}},
$$
and therefore, $r_{k+1}^t< k^{-1}$. Let $t'$ satisfy $t<t'<E$. Then 
$$
r_{k+1}^{t'}=(r_{k+1}^{t})^{t'/t}< k^{-t'/t}
$$
for any $k\geq L$.
Since $t'/t>1$, $\sum_kr_k^{t'}$ is a convergent series.
This contradicts the definition of $E$.
\end{proof}

Thus, for homogeneous residual sets we obtain the following analogue of Theorem \ref{th:2}:
\begin{cor}\label{cor:hom}
If the residual set $\cS$ of a packing 
is homogeneous and 
\bbe\label{eq:ohoho}
\dim_{lbox}\cS=\dim_{ubox}\cS=\dim_H\cS,
\ee
then the limit of $\log N(x)/\log x$ as $x\to\infty$ exists and
$$
\lim_{x\to\infty}\frac{\log N(x)}{\log x}=E=\dim_H \mathcal \cS.
$$
\end{cor}
\begin{rem}
There exist conditions on fractal sets that are easy to verify and that imply \eqref{eq:ohoho}  (see, e.g., conditions in Theorems 3 and 4 in \cite{F89}). In the next section we discuss another example of a homogeneous set satisfying \eqref{eq:ohoho}, namely 
a Sierpi\'nski carpet that is
the Julia set of a hyperbolic rational map. 
\end{rem}



\section{Julia sets of hyperbolic rational maps}\label{s:jls}
In this section we prove an analogue of Theorem \ref{th:2} for Sierpi\'nski carpets that are Julia sets of hyperbolic rational maps (Theorem \ref{th:Jul} below). Our proof is based
on the application of Corollary \ref{cor:hom} to Julia sets under consideration. For the background on the topics of complex dynamics used in this section see, e.g., \cite{Be}, \cite{CG}, and \cite{Mi2}.

\sbr

Let $f$  be a rational map of the Riemann sphere $\hat{\mathbb C}$. The \emph{Fatou set} $\cF(f)$ of $f$ is the set of all points $p$ in $\hat{\mathbb C}$ such that the family of iterates $\{f^k\}$ of $f$ is a normal family in some neighborhood of $p$. The \emph{Julia set} $\cJ(f)$ is the complement of the Fatou set in $\hat{\mathbb C}$. From the definitions one immediately sees that $\cF(f)$ is open and $\cJ(f)$ is compact. Moreover, $\cF(f)$ and $\cJ(f)$ are \emph{completely invariant} with respect to $f$, i.e., 
$$
f(\cF(f))=f^{-1}(\cF(f))=\cF(f),\quad f(\cJ(f))=f^{-1}(\cJ(f))=\cJ(f).
$$
If the Julia set $\cJ(f)$ of a rational map $f$ is a carpet (in which case we will refer to $\cJ(f)$ as {\it the carpet Julia set} $\cJ(f)$), then  $\cJ(f)$ is the residual set associated to the packing formed by the connected components of the Fatou set $\cF(f)=\cup_{k\in\mathbb N\cup\{0\}}D_k$. This follows from the definition of a residual set and the facts that $\cF(f)$ is open,  
$\cJ(f)\cap \cF(f)=\emptyset$, and $\cJ(f)\cup \cF(f)=\hat{\mathbb C}$.  The boundary curves $C_k=\partial D_k$ of these components are the peripheral circles of the carpet Julia set $\cJ(f)$.

\sbr

A rational map $f$ is said to be \emph{hyperbolic} if it is expanding in a neighborhood of its Julia set $\mathcal{J}(f)$ with respect to some conformal metric. More precisely, there exist a neighborhood $U$ of $\mathcal J(f)$, a smooth function $\lambda\colon U\to (0,\infty)$, and a constant $\rho>1$ such that
\begin{equation}\label{eq:hyp}
||f'(z)||_\lambda:=\frac{\lambda(f(z))\cdot ||f'(z)||}{\lambda(z)}\ge\rho,\quad\forall z\in U,
\end{equation}
where $||f'(z)||$ denotes the spherical derivative of $f$ at $z$. In what follows we will only consider Julia sets that are not equal to the whole sphere. Therefore, we may assume that they are compact subsets of the plane and the the spherical derivative in \eqref{eq:hyp}
 can be replaced by the modulus of the  derivative of $f$. 

\sbr

The function $\lambda$ can be used to define a new metric in a neighborhood of $\mathcal J(f)$ as follows. For a  rectifiable path $\gamma$ in $U$ let
$$
{\rm length}_\lambda(\gamma):=\int_\gamma\lambda\, ds.
$$
From~\eqref{eq:hyp} we  immediately get
$$
{\rm length}_\lambda(f\circ\gamma)\ge\rho\, {\rm length}_\lambda(\gamma)
$$
for any rectifiable path $\gamma$ in $U$ such that $f\circ\gamma$ is also in $U$. 
Therefore, a hyperbolic rational map $f$ is locally expanding by a factor at least $\rho$ with respect to the path metric defined  in a small neighborhood of $\mathcal J(f)$ by the formula
$$
d_\lambda(z,w):=\inf_\gamma{\rm length}_\lambda(\gamma),
$$
where $\gamma$ runs through all the rectifiable paths in $U$ connecting $z$ and $w$. Below we drop the index $\lambda$ and write $d(z,w)$  instead of  $d_\lambda(z,w)$. Also, in what follows $\diam$ refers to the diameter with respect to $d_\lambda$. Since $\lambda$ is  bounded away from $0$ and $\infty$ in a neighborhood of $\mathcal J(f)$, the metric $d_\lambda$ is locally comparable (i.e., bi-Lipschitz equivalent) to the Euclidean metric in such a neighborhood. Therefore, it is irrelevant which metric, the Euclidean metric or the metric $d_\lambda$, one uses for verification of~\eqref{eq:cond1}, \eqref{eq:cond2}, and~\eqref{eq:cond2.5}.

\begin{thm}\label{th:Jul}
Assume that $f$ is a hyperbolic rational map whose Julia set $\mathcal J(f)$ is a Sierpi\'nski carpet. Then 
$$
\lim_{x\to\infty}\frac{\log N(x)}{\log x}=E=\dim_H \mathcal J(f),
$$
where $N$ is the curvature distribution function and $E$ is the exponent of the packing by the  connected components of the Fatou set of $f$, and $\dim_H \mathcal J(f)$ is the Hausdorff dimension of $\mathcal J(f)$.
\end{thm}
\begin{proof}
We first show that if $f$ is a hyperbolic rational map whose Julia set $\mathcal J(f)$ is a Sierpi\'nski carpet, then  $\mathcal J(f)$ is homogeneous. 

\sbr

According to~\cite[(3.6)]{F89}, for a small enough $r>0$ and a ball $B(p,r)$ centered at $p\in\mathcal J(f)$ there exists a natural number $n$ so that  the iterate $f^n|_{B(p,r)}$ is bi-Lipschitz  when the metric on $B(p,r)$ is rescaled by $1/r$. More precisely, there exist positive constants $a, b, r_0$ such that for every $p\in\mathcal J(f)$ and $0<r\le r_0$ there exists $n\in\mathbb N$ with 
\begin{equation}\label{eq:biLip}
a\frac{d(z,w)}{r}\le d(f^n(z),f^n(w))\le b\frac{d(z,w)}{r},\quad \forall z, w\in B(p,r).
\end{equation}
Note that condition \eqref{eq:biLip} means that the Julia set $\cJ(f)$ of a hyperbolic rational map $f$ is {\em approximately self-similar} according to Definition 3.11 in \cite{bk}.

\sbr

Now we are ready to prove that $\mathcal J(f)$ is homogeneous, namely it satisfies~\eqref{eq:cond1}, \eqref{eq:cond2}, and~\eqref{eq:cond2.5}. Let $\{C_k\}$ denote the sequence of peripheral circles of $\mathcal J(f)$.
First of all, according to Whyburn's characterization~\cite{W58} we have ${\rm diam}\, C_k\to 0$ as $k\to\infty$. Also, since every peripheral circle $C$  of $\mathcal J(f)$ is the boundary of a  connected component of $\cF(f)$, the curve $f(C)$ is also a peripheral circle of $\cJ(f)$. 

\sbr

To show~\eqref{eq:cond1} assume that $C$ is an arbitrary peripheral circle of $\mathcal J(f)$ whose diameter is at most $\frac{a}{b}\min(r_0, a)$. 
We set $r=\frac{b}{a}\diam\, C$ and choose an arbitrary $p\in C$.  
Note that $C\subset B(p,r)$, since we may assume that $b/a>1$.
Because $r\le r_0$ and $p\in\cJ(f)$,  
by~\eqref{eq:biLip} there exists $n\in\mathbb N$ such that
\bbe\label{eq:vspom1}
\frac{a^2}{b}\leq \diam\,f^n(C)\leq a.
\ee
Since there are finitely many peripheral circles whose diameter is at least  $\frac{a}{b}\min(r_0, a)$, there exists a constant $\alpha\ge 1$ such that for each of these ``large" curves $C_k$ there  exist inscribed and circumscribed concentric circles centered at $p_k$ of radii $r_k$ and $R_k$, respectively, with
$
{R_k}/{r_k}\le\alpha.
$
Due to \eqref{eq:vspom1} the peripheral circle $f^n(C)$ is one of the large peripheral circles, say $C_k$. 
Let $q\in B(p,r)$ satisfy $f^n(q)=p_k$.
Now~\eqref{eq:biLip} applied to $B(p,r)$ with $z=q$ and $w\in C$ implies that a disc centered at $q$ of radius at least $rr_k/b$ can be inscribed in $C$ and
a disc centered at $q$ of radius at most $rR_k/a$ can be circumscribed around $C$.
Therefore, the quotient of the circumscribed radius to the inscribed one for $C$ is at most $b\alpha/a$ and~\eqref{eq:cond1} follows. 

\sbr

We argue similarly to prove~\eqref{eq:cond2}. Suppose that $p$ is an arbitrary point in $\mathcal J(f)$ and $0<r\le\diam\, \mathcal J(f)$. If $r\ge\min(r_0,a/2)$, then the fact that $\mathcal J(f)$ is a compact set with no interior points implies the existence of $\tilde\beta$ such that any disc $B(p,r)$ intersects a peripheral circle $C_k$ of $\mathcal J(f)$ with
$$
\frac{r}{\tilde\beta}\le\diam\, C_k\le\tilde\beta r.
$$
Moreover, by choosing $\tilde\beta$ sufficiently large, we can require that the diameter of each such $C_k$ is at most $a/2$.

\sbr

We now assume that $r<\min(r_0,a/2)$. By the first inequality in~\eqref{eq:biLip}, there exists $n\in\mathbb N$ such that $f^n(B(p,r))$ contains an open disc centered at $f^n(p)$ of radius at least $a$. It follows from the above that $B(f^n(p),a/2)$ intersects a peripheral circle $C_k$ of $\mathcal J(f)$ with 
\begin{equation}\label{eq:perc}
\frac{a}{2\tilde\beta}\le\diam\, C_k\le \frac{a}2.
\end{equation}
In particular, $C_k$ is contained in $B(f^n(p), a)$ and, consequently, in $f^n(B(p,r))$. Therefore, there exists a peripheral circle $C$ of $\mathcal J(f)$ that intersects $B(p,r)$ and such that $f^n(C)=C_k$. Moreover, \eqref{eq:biLip} combined with~\eqref{eq:perc} yield
$$
\frac{a}{2b\tilde\beta}r\le\diam\, C\le\frac12 r.
$$
The proof of~\eqref{eq:cond2} (with $\be=\max(1/2,\tilde\be,2b\tilde\be/a)$) is thus complete.

\sbr

Finally, to prove~\eqref{eq:cond2.5} we argue by contradiction. Assume that there exists a sequence of pairs $\{C_{j_i}, C_{k_i}\}_i$ of distinct peripheral circles of $\mathcal J(f)$ such that $\Delta(C_{j_i}, C_{k_i})\to0$ as $i\to\infty$. By symmetry we may assume that $\diam\, C_{j_i}\le\diam\, C_{k_i}$ for all $i$. Since there are only finitely many peripheral circles of diameter at least a given number, we must have $\diam\, C_{j_i}\to0$ 
as $i\to\infty$. Also, 
$$
{\rm dist}(C_{j_i}, C_{k_i})=\Delta(C_{j_i}, C_{k_i})\diam\, C_{j_i}\to0,\quad i\to\infty.
$$
Let $p_i$ be a point on $C_{j_i}\subset\mathcal J(f)$  and let $q_i$ be a point on $C_{k_i}$ such that 
$${\rm dist}(C_{j_i}, C_{k_i})=d(p_i, q_i).$$
Let $0<\epsilon<1$ be arbitrary and let $i$ be so large that 
$$
r:=2\diam\, C_{j_i}\le {r_0}\quad\text{ and }\quad \Delta(C_{j_i}, C_{k_i})<\epsilon.
$$
Then $d(p_i,q_i)<r/2$, hence $q_i\in B(p_i, r)$ and there is at least one more point $q'\ne q_i$ in the intersection $C_{k_i}\cap B(p_i, r)$.
By applying 
\eqref{eq:biLip} to $B(p_i, r)$ with $z,w\in C_{j_i}\subset B(p_i, r)$, we conclude that there exists $n\in\mathbb N$ such that for the peripheral circle $C_j=f^n(C_{j_i})$ we have $\diam\, C_j\geq a/2$. Similarly, by applying \eqref{eq:biLip} to $B(p_i, r)$ with $z=q'$ and $w=q_i$, we conclude that for the peripheral circle $C_k=f^n(C_{k_i})$ we have $\diam\, C_k\ge a/2$ (recall that $\diam\, C_{k_i}\ge \diam\, C_{j_i}$). Finally, \eqref{eq:biLip} applied to $B(p_i, r)$ with $z=p_i$ and $w=q_i$ gives
$$
\Delta(C_j, C_k)\le 2b\frac{{\rm dist}(C_{j_i}, C_{k_i})}{ra}=2b\frac{\Delta(C_{j_i}, C_{k_i})\diam\, C_{j_i}}{ra}=\frac{b}{a}\Delta(C_{j_i}, C_{k_i})<\frac{b}{a}\epsilon.
$$
This is  a contradiction because $\epsilon$ is arbitrary and there are only finitely many pairwise disjoint peripheral circles $C$ with $\diam\, C\ge a/2$. Hence, \eqref{eq:cond2.5} follows.

\sbr

The rest of the proof is a simple application of the results of~\cite{F89} and the previous results of this paper.
Indeed, by~\cite[Theorem in \S3]{F89}, if $\mathcal{J}(f)$ is the Julia set of a hyperbolic rational map $f$, then 
$$
\dim_{lbox} \mathcal J(f)=\dim_{ubox} \mathcal J(f)=\dim_H \mathcal J(f).
$$
Corollary \ref{cor:hom} therefore implies that the limit of $\log N(x)/\log x$ as $x\to\infty$ exists and
$$
\lim_{x\to\infty}\frac{\log N(x)}{\log x}=E=\dim_H \mathcal J(f).
$$
This completes the proof of Theorem~\ref{th:Jul}.
\end{proof}

\begin{rem}
The peripheral circles of $\mathcal J(f)$ as above are in fact uniform quasicircles, and hence $\mathcal J(f)$ also satisfies~\eqref{eq:cond3.5} (see~\cite[Corollary~2.3]{S95}).
\end{rem}

\section{Concluding remarks and further questions}\label{CR}

In their recent work \cite{k} A.~Kontorovich and H.~Oh proved a result about Apollonian circle packings that strengthens Theorem~\ref{th:2}. Namely, that $\lim_{x\to\infty}({N(x)}/{x^\al})=c$ for some positive constant $c$, where $\al=E=\dim_H\cS$ (in the notation of Theorem~\ref{th:2}). We believe that an analogue of this result also should hold in our context, i.e., for packings whose associated residual sets are homogeneous and approximately self-similar. 
More specifically, 
we make the following conjecture.  
\begin{conjecture}
If $f$ is a hyperbolic rational map whose Julia set is a Sierpi\'nski  carpet, then
there exists a positive constant $c$ such that
\bbe\label{eq:kon}
N(x)\sim c\cdot x^{E},
\ee
where $N$ is the curvature distribution function and $E$  is the exponent of the packing by the connected components of the Fatou set of $f$. In other words, $\lim_{x\to\infty}({N(x)}/{x^E})=c.$
\end{conjecture}

\sbr


A homeomorphism $f\colon (X,d_X)\to(Y,d_Y)$ between two metric spaces is called \emph{quasisymmetric} if there exists a homeomorphism $\eta\colon [0,\infty)\to[0,\infty)$ such that
$$
\frac{d_Y(f(x),f(x'))}{d_Y(f(x),f(x''))}\le\eta\left(\frac{d_X(x,x')}{d_X(x, x'')}\right)
$$
for any three distinct points $x,x'$, and $x''$ in $X$. It turns out that quasisymmetric maps between subsets of the plane or the sphere preserve properties~\eqref{eq:cond1}, \eqref{eq:cond2}, \eqref{eq:cond2.5}, and~\eqref{eq:cond3.5} up to a change of constants;
see~\cite{H01} for background on quasisymmetric maps. An important invariant for quasisymmetric maps is the \emph{conformal dimension}. For a metric space $\cS$ this is the infimum of the Hausdorff dimensions of all images of $\cS$ under quasisymmetric maps. The conformal dimension of a metric space is always at least its topological dimension and, trivially, at most the Hausdorff dimension.
For example, the conformal dimension of the standard Cantor set $\mathcal C$ is zero, which is its topological dimension, and is strictly less than its Hausdorff dimension $\log2/\log3$. The infimum in the definition of the conformal dimension of $\cC$ is not achieved though, i.e., there is no metric space of Hausdorff dimension 0 that is quasisymmetric to $\cC$.

\sbr

The value of the conformal dimension of the Sierpi\'nski carpet $S_3$ is unknown (see~\cite[15.22 Open problem]{H01}). However, it has a non-trivial lower bound $1+\log2/\log3$, which is strictly greater than its topological dimension 1. If in the definition of conformal dimension we fix the target to be the plane, it is interesting to see what will be the shapes of the complementary components of quasisymmetric images of $S_3$ that are near-optimal with respect to the conformal dimension. Proposition~\ref{p:dual} implies that the reciprocal diameters of peripheral circles of near-optimal images of $S_3$ must go to infinity faster than those for $S_3$. At present we do not know in what way the quasisymmetric map should distort the boundaries of the complementary components to make the Hausdorff dimension smaller.

\end{document}